\documentclass[final]{siamltex}

\usepackage{graphics}
\usepackage{amssymb, latexsym, amsmath, epsfig, epstopdf}
\usepackage{graphicx}

\newtheorem{remark}[theorem]{Remark}

\def\cal{\mathcal}
\def\btd{{\bigtriangledown}}

\def\E{{\mathcal E}}

\def\O{\Omega}

\newcommand\ddelta\bigtriangledown
\newcommand\ld\lambda
\newcommand\Ld\Lambda

\def \bP {\Bbb P}

\def \bZ {\Bbb Z}

\begin{document}
\title{Finite volume schemes of any order \\
on rectangular meshes}

\author{Zhimin Zhang
\thanks{Department of Mathematics, Wayne State University, Detroit, MI 48202, USA.
     This author is supported in part by the US National Science Foundation through grant
     DMS-111530, the Ministry of Education of China through the Changjiang Scholars program,
    and Guangdong Provincial Government of China through the
   ``Computational Science Innovative Research Team" program.}
\and Qingsong Zou
\thanks{Corresponding author, College of Mathematics and Scientific Computing and Guangdong Province Key Laboratory of Computational Science, Sun Yat-sen
          University, Guangzhou, 510275, P. R. China. This author is supported in part by
          the National Natural Science Foundation of China under the grant 11171359 and in part by the Fundamental Research Funds for the Central
          Universities of China.}
}
\maketitle
\begin{abstract}
In this paper, we analyze vertex-centered finite volume method (FVM) of any order for elliptic equations on rectangular meshes.
The novelty is a unified proof of the inf-sup condition, based on which, we show that the FVM approximation converges to the exact solution with the optimal rate in the energy norm.
Furthermore, we discuss superconvergence property of the FVM solution.
With the help of this superconvergence result, we find that the FVM solution also converges to the exact solution with the optimal rate in the $L^2$-norm.
Finally, we validate our theory with several numerical experiments.



\vskip .3cm

\end{abstract}

\section{Introduction}
\setcounter{equation}{0}

\setcounter{equation}{0}
During the past several decades, the {\it finite volume method} (FVM) has attracted much attention. We refer to \cite{Bank.R;Rose.D1987}-\cite{ChenWuXu2011}, \cite{Emonot1992}-\cite{HymanKnappScovel1992}, \cite{Nocolaides1995}-\cite{C.Shu2003}, \cite{Xu.J.Zou.Q2009, zou2009a}  for an incomplete list of references.
Due to its local conservation of numerical fluxes and other advantages, the FVM
is very popular in scientific and engineering computations, especially in computational fluid dynamics, see, e.g., \cite{Emonot1992,HymanKnappScovel1992,Lazarov1996} and \cite{Nocolaides1995}-\cite{C.Shu2003}.

Comparing to its wide applications, the mathematical theory of FVM  (cf., \cite{Barth.T;Ohlberger2004, EymardGallouetHerbin2000, LeVeque2002, Li.R2000})
 has not been fully developed, at least,  not as satisfactory as that for the finite element method (FEM).
In fact, since the FV schemes depend heavily on the underlying meshes, the error analysis in the literature was often done case-by-case.
For instance, the linear FV scheme can be regarded as a small perturbation of its corresponding linear FE scheme,
whose convergence properties  have been well studied, see e.g., \cite{Bank.R;Rose.D1987, Cai.Z1991, Ewing.R;Lin.T;Lin.Y2002, Hackbusch.W1989a}. On the other hand,
high-order FV schemes are substantially different from their corresponding FE schemes, therefore only a few special high-order schemes have been studied,
 see \cite{ Cai.Z_Park.M2003, Chen.L, ChenWuXu2011, Liebau1996, Plexousakis_2004, Tian.Chen_1991, Xu.J.Zou.Q2009}. So far, we have not seen analysis for FV schemes of an arbitrary order.

 In this paper, we provide a unified analysis for vertex-centered FV schemes of any order on rectangular meshes.  We construct our FV schemes under the framework of the Petrov-Galerkin method by letting the trial space be the Lagrange finite element space  with the interpolation points being the Lobatto points  and by constructing control volumes  with the Gauss points  in a rectangular element.

  It is known that the proof of the stability  is a challenging task in the error analysis of FV schemes. Earlier works (see, e.g. \cite{Liebau1996, Li.R2000,Xu.J.Zou.Q2009,ChenWuXu2011}) utilized {\it element stiffness matrix analysis} for this task. The {\it element stiffness matrix analysis} often requires to calculate all
   eigenvalues of an element stiffness matrix and thus is difficult to be generalized to schemes
  of any order. Our new approach in this paper for proving the stability (or in general the inf-sup condition) is different from the {\it element stiffness matrix analysis}.
  A novel and non-traditional global mapping from the trial space to the test space is introduced.
  This mapping avoids calculating eigenvalues of an element stiffness matrix
  and makes the establishment of the global inf-sup condition for any order possible.
  An interesting feature is that when the coefficient $\alpha$ in (\ref{eq:2.1}) is a piecewise constant function
  with respect to the underlying mesh, the inf-sup condition is uniformly valid for any mesh size $h$, i.e.,
  there is no requirement ``for sufficiently small $h$".
  In particular, for the Poisson equation, the inf-sup condition is uniformly valid for any mesh size $h$.
  Once the inf-sup property has been established, the error analysis in the energy norm is then a routine work.

Another feature of this work is the superconvergence analysis. We prove that the FV solution $u_{\cal P}$ is super-close to the Lobatto interpolant $u_I$ of the exact solution, namely, $|u_{\cal P} - u_I|_1$ converges one order higher than the optimal rate. The result simulates the counterpart result in the FEM. A by-product of this superconvergence result is the optimal $L^2$ error estimate. Conventionally, the $L^2$ error estimate is accomplished by the duality argument or the so-called Aubin-Nitsche trick. Unfortunately, this technique is very difficult to be used in our case for higher-order FVM. The adoption of the superconvergence analysis avoids this difficulty.

    We organize the rest of the paper as follows.
    In Section 2 we present  FV schemes of any order for elliptic equations on rectangular meshes.
    In Section 3 we provide convergence analysis and establish the optimal convergence rate in both
    $H^1$ and $L^2$ norms. The superconvergence property of the FVM solution has also been studied in this section.
    Next, numerical examples are provided in Section 4 to confirm our theory. And lastly, some concluding remarks are given in Section 5.

  In the rest of this paper, ``$A\lesssim B$" means that $A$ can be
  bounded by $B$ multiplied by a constant which is independent of
  the parameters which $A$ and $B$ may depend on. ``$A\sim B$" means $``A\lesssim B"$ and $``B\lesssim A"$.

\bigskip
\section{FVM Schemes of Any Order}
\setcounter{equation}{0}

In this section, we present finite volume schemes of any order to solve the following second-order elliptic boundary value problem
\begin{eqnarray}\label{eq:2.1}
-\btd\cdot(\alpha \btd u) =&f  &\ \ {\rm in}\ \ \ \Omega,\\ \label{eq:2.2}
      u=&0 &\ \ {\rm on}\ \ \  \Gamma,
\end{eqnarray}
where  $\Omega=[a,b]\times [c,d]$ is a rectangle, $\Gamma=\partial\Omega$, $\alpha\in L^\infty$ and it is bounded from below: There exists a constant $\alpha_0>0$ such that $\alpha(x)\geq \alpha_0$ for almost all $x\in \Omega$,  and  $f$ is a real-valued function defined on $\Omega$.

We present our finite volume schemes in the framework of Petrov-Galerkin method. We first construct the primal partition ${\cal P}$ and the trial space.
 Let $a=x_0< x_1<\ldots<x_m=b,\  c=y_0< y_1<\ldots<y_n=b$. For a positive integer $k$, let $\bZ_k=\{1,\ldots,k\}$ and $\bZ_k^0=\{0,1,\ldots,k\}$.
 For all $i\in \bZ_m, j\in \bZ_n$, let $h_i^x=x_i-x_{i-1},h_j^y=y_j-y_{j-1}$ and
  $h=\max\left\{\max (h^x_i,h_j^y)\big| (i,j)\in \bZ_{m,n}\right\}$ where $\bZ_{m,n}=\bZ_m\times \bZ_n$. We denote the associated partition of $\Omega$ by
\[
  {\cal P}= \left\{\tau_{i;j}\big| (i,j)\in\bZ_{m,n}\right\}
\]
where  $\tau_{i;j}=[x_{i-1},x_i]\times [y_{j-1},y_j].$ We choose the trial space as the standard FEM space defined by
\[
   U^r_{\cal P}=\{v\in C(\Omega): v|_{\tau}\in {\mathbb Q}_r, \forall\tau\in{\cal P}, v|_{\partial\Omega}=0\},
\]
  where ${\mathbb Q}_r$ is the set of all bi-polynomials of degree no more than $r$. Obviously, $\dim U^r_{\cal P}=(mr-1)(nr-1)$.

 We next describe the dual partition and the test space.  Let $G_1,\ldots,G_r$ be $r$ Gauss points, i.e., zeros of the Legendre polynomial of $r$th degree,
   on the interval $[-1,1]$.
For any given $(i,j)\in \bZ_{m,n}$, let
\[
  g^x_{i,k}=\frac {1}{2}(x_i+x_{i-1}+h_i^xG_k),\  g^y_{j,l}=\frac {1}{2}(y_j+y_{j-1}+h_j^y G_l);\ \forall k,l\in \bZ_r
\]
be the Gauss points in the interval $[x_{i-1},x_i]$ and $[y_{j-1},y_j]$, respectively. Let
\[
  \hat{r}_i^x=\left\{\begin{array}{lll}
 r & \text{if} &i\in \bZ_m,\\
 1 & \text{if} &i=0,
\end{array}
\right. \quad  \hat{r}_j^y=\left\{\begin{array}{lll}
 r & \text{if} &j\in \bZ_n,\\
1 & \text{if} &j=0,
\end{array}
\right.
\]
and
\[
g^x_{0,1}=a,\  g^y_{0,1}=c;\ g^x_{m, r+1}=b,\  g^y_{n,r+1}=d;
\]
and for all $i\in \bZ_{m-1}, j\in \bZ_{n-1}$
\[
g^x_{i,\hat{r}_i^x+1}=g^x_{i+1,1},\quad   g^y_{j,\hat{r}_j^y+1}=g^y_{j+1,1}.
\]

  With these Gauss points,  we construct a dual partition
\[
{\cal P}'=\{\tau'_{i,k;j,l}\big| (i,k,j,l)\in {\cal Z}_0\},
\]
where ${\cal Z}_0=\bZ^0_m\times\bZ_{\hat{r}_i^x}\times\bZ^0_n\times\bZ_{\hat{r}_j^y}$ and the   control volume
\[
\tau'_{i,k;j,l}:=[g^x_{i,k}, g^x_{i,k+1}]\times [g^y_{j,l}, g^y_{j,l+1}].
\]
The test space $V_{\cal P'}$ consists of the piecewise constants with respect
to the partition ${\cal P'}$  which vanishes on the boundary control volumes.
In other words,
\[
V_{\cal P'}=\text{Span}\left\{\psi_{i,k;j,l}\big| (i,k;j,l)\in {\cal Z}_1\right\}
\]
where $\psi_{i,k;j,l}=\chi_{\tau'_{i,k;j,l}} $ is the characteristic function on the control volume $\tau'_{i,k;j,l}$,
${\cal Z}_1=\bZ_m\times\bZ_{r_i^x}\times\bZ_n\times\bZ_{r_j^y}$
 and
\[
  r_i^x=\left\{\begin{array}{lll}
 r & \text{if} &i\in \bZ_{m-1},\\
 r-1 & \text{if} &i=m,
\end{array}
\right. \quad  r_j^y=\left\{\begin{array}{lll}
 r & \text{if} &j\in \bZ_{n-1},\\
r-1 & \text{if} &j=n.
\end{array}
\right.
\]
Note that $\dim V_{\cal P'}=(mr-1)\times (nr-1)=\dim U^r_{\cal P}$.

We are now ready to present our finite volume schemes. The finite volume solution of \eqref{eq:2.1} and \eqref{eq:2.2}
is a function $u_{\cal P}\in U^r_{\cal P}$ which satisfies the following conservation law
\begin{equation}\label{conserve}
 -\int_{\partial\tau'_{i,k;j,l}} \alpha\frac{\partial u_{\cal P}}{\partial {\bf n}}ds =\int_{\tau'_{i,k;j,l}}fdxdy
\end{equation}
on each control volume $\tau'_{i,k;j,l},  (i,k;j,l)\in {\cal Z}_1$, where ${\bf n}$ is the unit outward normal on the boundary curve $\partial\tau'_{i,k;j,l}$.  Let $w_{\cal P'}\in V_{\cal P'}$,  $w_{\cal P'}$ can be written as
\[
w_{\cal P'}=\sum_{(i,k;j,l)\in {\cal Z}_1}w_{i,k;j,l}\psi_{i,k;j,l}
\]
 where the coefficients $w_{i,k;j,l}$ are constants. Multiplying \eqref{conserve} with $w_{i,k;j,l}$ and then summing up for all $i,k,j,l$, we obtain
\[
 -\sum_{(i,k;j,l)\in {\cal Z}_1} w_{i,k;j,l}\int_{\partial\tau'_{i,k;j,l}}\alpha \frac{\partial u_{\cal P}}{\partial {\bf n}}ds=\int_\O f w_{\cal P'}dxdy.
\]

Defining the FVM bilinear form for all $v\in H^1_0(\Omega), w_{\cal P'}\in V_{\cal P'}$ as
\begin{equation}\label{biform}
   a_{\cal P}(v,w_{\cal P'})= -\sum_{(i,k;j,l)\in {\cal Z}_1} w_{i,k;j,l}\int_{\partial\tau'_{i,k;j,l}}\alpha \frac{\partial v}{\partial {\bf n}}ds,
\end{equation}
the finite volume method for solving equations \eqref{eq:2.1} and \eqref{eq:2.2} reads as :  Find $u_{\cal P}\in U^r_{\cal P}$ such that
\begin{equation}\label{bilinear}
   a_{\cal P}(u_{\cal P},w_{\cal P'})=(f,w_{\cal P'}),\ \ \forall w_{\cal P'}\in V_{\cal P '}.
\end{equation}

Noticing that ${\cal Z}_1\subset {\cal Z}_0$, a function $w_{\cal P'}\in V_{\cal P'}$
can also be written as
\[
w_{\cal P'}=\sum_{(i,k;j,l)\in {\cal Z}_0}w_{i,k;j,l}\psi_{i,k;j,l},
\]
if we let
\[
w_{i,k;j,l}=0, \forall (i,k;j,l)\in {\cal Z}_0\setminus {\cal Z}_1.
\]
For all $(i,k;j,l)\in {\cal Z}_2=\bZ_m\times \bZ_r\times \bZ_n\times \bZ_r$,
we define the jumps of $w_{\cal P'}$ in the $x$-direction and $y$-direction  as
\[
   [w_{i,k;j,l}]^x=w_{i,k;j,l}-w_{i,k-1;j,l},\quad
  [w_{i,k;j,l}]^y=w_{i,k;j,l}-w_{i,k;j,l-1},
\]
respectively.

A simple calculation yields that
\begin{eqnarray}\nonumber
   a_{\cal P}(v,w_{\cal P'})
&=& \sum_{(i,k;j,l)\in {\cal Z}_3}[w_{i,k;j,l}]^x\int_{g^y_{j,l}}^{g^y_{j,l+1}}\alpha(g^x_{i,k}, y)\frac{\partial v}{\partial x}(g^x_{i,k}, y)dy\\
&+& \sum_{(i,k;j,l)\in {\cal Z}_4}[w_{i,k;j,l}]^y\int_{g^x_{i,k}}^{g^x_{i,k+1}}\alpha(x,g^y_{j,l})\frac{\partial v}{\partial y}(x,g^y_{j,l})dx,
\end{eqnarray}
where ${\cal Z}_3=\bZ_m\times \bZ_r\times\bZ_n\times\bZ_{r^y_j}$ and ${\cal Z}_4=\bZ_m\times \bZ_{r_i^x}\times\bZ_n\times\bZ_r$.
\bigskip

\section{Error Analysis}
\setcounter{equation}{0}

The error analysis of FVM can also be done under the framework of Petrov--Galerkin methods, see  \cite{Bank.R;Rose.D1987},  \cite{Li.R2000}, and \cite{Xu.J.Zou.Q2009}.
Following this approach, our FVMs error analysis requires the study of the continuity (boundedness) and inf-sup property of the FVM bilinear form.

\subsection{Continuity}
 Let $\E_{\cal P'}$ be the set of interior edges of the dual partition ${\cal P'}$. Then for all\ $v\in\ H_0^1(\Omega),\ w_{\cal P'}\in \ V_{\cal P'}$,
\begin{eqnarray} \label{bilinear1}
a_{\cal P}(v,w_{\cal P'})&=& \sum_{E\in \E_{\cal P'}}[w_{\cal P'}]\int_E \alpha\frac{\partial v}{\partial {\bf n}}ds,
\end{eqnarray}
where  $[w_{\cal P'}] = w_{\cal P'}|_{\tau_2} - w_{\cal P'}|_{\tau_1}$ across the common edge $E=\tau_1\cap \tau_2$ of two rectangles $\tau_1, \tau_2 \in {\cal P'}$ and ${\bf n}$ denotes the normal vector on $E$ pointing from $\tau_1$ to $\tau_2$.

To study the continuity of $a(\cdot,\cdot)$ ,  we define a semi-norm in the test space $V_{\cal P'}$ for all $w_{\cal P'}\in V_{\cal P'}$
by
\[
\big|w_{\cal P'}\big|_{\cal P'}=\left(\sum_{E\in {\cal E}_{\cal P'}}
h_E^{-1}\int_E \left[w_{\cal P'}\right]^2ds\right)^{1\over 2},
\]
where  $h_E$ is the diameter
of an edge $E$, and  a semi-norm in the so-called {\it broken $H^2$ space}
\[
   H_{\cal P}^{2}(\Omega)=\{v\in C(\Omega): v|_{\tau}\in H^2, \forall\tau\in{\cal P}\}
\]
 for all $ v\in H^2_{\cal P}(\Omega)$ by
\[
  \big|v\big|_{\cal P}=\left(\sum_{\tau\in\cal P}|v|_{1,\tau}^2+h_\tau^2|v|_{2,\tau}^2\right)^{\frac 12},
\]
where $h_\tau$ is the diameter of $\tau$.
Note that the mesh-dependent semi-norm $\big|\cdot\big|_{\cal P}$ has been used in
the discontinuous Galerkin method (cf., \cite{ArnoldBrezziCockburnMarini2003}) and was introduced first into the FVM in \cite{Xu.J.Zou.Q2009}.

\begin{theorem}\label{thm:continuity}
The finite volume bilinear form $a_{\cal P}(\cdot,\cdot)$
 is variationally exact:
 \begin{equation}\label{exact}
a_{\cal P}(u,w_{\cal P'})=(f,w_{\cal P'})\ \ \ \ \forall\  w_{\cal P'}\in \ V_{\cal P'}
  \end{equation}
 and continuous : for all\ $v\in\ H_0^1(\Omega)\cap H_{\cal P}^{2}(\Omega),\ w_{\cal P'}\in \ V_{\cal P'}$,
\begin{equation}\label{bounded}
|a_{\cal P}(v,w_{\cal P'})|\le M |v|_{\cal P} |w_{\cal P'}|_{\cal P'}
\end{equation}
  where the constant $M>0$  depends only on $\alpha$ and $r$.
\end{theorem}
\begin{proof}
 First, \eqref{exact} follows  by multiplying \eqref{eq:2.1} with an arbitrary function $w_{\cal P'}\in V_{\cal P'} $
and then using Green's  formula in each control volume $\tau\in {\cal P}'$.

Next we prove \eqref{bounded}.  By the Cauchy-Schwartz inequality, for all\ $v\in\ H_0^1(\Omega),\ w_{\cal P'}\in \ V_{\cal P'}$,
we have
\[
a_{\cal P}(v,w_{\cal P'})
\le \|\alpha\|_\infty|w_{\cal P'}|_{\cal P'} \left(\sum_{E\in \E_{\cal P'}} h_E\int_E \left(\frac{\partial v}{\partial {\bf n}}\right)^2ds\right)^{\frac12}.
\]
By the trace inequality and the shape regularity of ${\cal P},$
\[
\left(h_E\int_E \left(\frac{\partial v}{\partial {\bf n}}\right)^2ds\right)^{\frac12}\lesssim |v|_{1,\tau_1\cup\tau_2}+ h_E|v|_{2,\tau_1\cup\tau_2}
\]
where $\tau_1,\tau_2\in {\cal P}'$ are  two control volumes  sharing the common edge $E$.
Therefore,
\begin{eqnarray*}
a_{\cal P}(v,w_{\cal P'})
&\lesssim& |w_{\cal P'}|_{\cal P'} \left(\sum_{E\in \E_{\cal P'}} |v|^2_{1,\tau_1\cup\tau_2}+ h_E^2|v|^2_{2,\tau_1\cup\tau_2}\right)^{\frac12}\\
&\lesssim& |w_{\cal P'}|_{\cal P'} \left(\sum_{\tau\in {\cal P}} |v|^2_{1,\tau}+ h_\tau^2|v|^2_{2,\tau}\right)^{\frac12}.
\end{eqnarray*}
Then the boundedness \eqref{bounded} is proved.
\end{proof}

\bigskip

\subsection{Inf-sup condition}
This subsection is the core of the paper. The analysis here is technical, and yet, it is new and no-traditional.
A key step is the introduction of a global projection (\ref{piv}) based on the idea of the Gauss quadrature.
We begin with some definitions and  notations. We denote by $A_j, j\in \bZ_r$ the weights of the Gauss quadrature
\[
Q_r(F)=\sum_{j=1}^r A_j F(G_j)
\]
for computing the integral
\[
I(F)=\int_{-1}^1 F(x) dx.
\]
It is well-known that $Q_r(F)=I(F)$ for all $F\in \bP_{2r-1}(-1,1)$.
For any given $(i,j)\in \bZ_{m,n}$, let
\[
  A^x_{i,k}=\frac{1}{2}h^x_iA_k,\  A^y_{j,k}=\frac {1}{2}h^y_jA_k,\  k\in \bZ_r
\]
be the Gauss weights in the intervals $[x_{i-1},x_i]$ and $[y_{j-1},y_j]$, respectively.
On the other hand,  $[w_{i,k;j,l}]^x$ and  $[w_{i,k;j,l}]^y$ are well defined for all $(i,k;j,l)\in {\cal Z}_2$ and we have
\[
[[w_{i,k;j,l}]^x]^y=[[w_{i,k;j,l}]^y]^x=w_{i,k;j,l}+w_{i,k-1;j,l-1}-w_{i,k-1;j,l}-w_{i,k;j,l-1}.
\]
We denote
\[
\lfloor w \rfloor_{i,k;j,l}=w_{i,k;j,l}+w_{i,k-1;j,l-1}-w_{i,k-1;j,l}-w_{i,k;j,l-1}, \forall (i,k;j,l)\in {\cal Z}_2.
\]

We define the mapping from the the trial space  $U_{\cal P}^r$ to the test space $V_{\cal P'}$ as :
\begin{equation}\label{piv}
\Pi v_{\cal P}=\sum_{(i,k,j,l)\in {\cal Z}_1} (\Pi v_{\cal P})_{i,k;j,l}\psi_{i,k;j,l}\in V_{\cal P'}, v_{\cal P}\in U_{\cal P}^r,
\end{equation}
where the coefficients $(\Pi v_{\cal P})_{i,k;j,l}$ are determined by the constraints
\begin{equation}\label{mapcond}
\lfloor \Pi v_{\cal P} \rfloor_{i,k;j,l}=A_{i,k}^xA_{j,l}^y \frac{\partial^2 v_{\cal P}}{\partial x\partial y}(g_{i,k}^x, g^y_{j,l}), \forall (i,k;j,l)\in {\cal Z}_1.
\end{equation}
%

\begin{remark}We can not suppose a priori that \eqref{mapcond} holds for all $(i,k;j,l)\in {\cal Z}_2$, because $\#{\cal Z}_2=mnr^2$ which is greater than
$\dim V_{\cal P'}=(mr-1)(nr-1)$. However, $\#{\cal Z}_1=(mr-1)(nr-1)$, so we constraint \eqref{mapcond} only for the indices in ${\cal Z}_1$.
We present the Gauss points corresponding the index sets ${\cal Z}_1$ and ${\cal Z}_2$ in Figure  \ref{z1z2} ($r=2$). In this figure, the gauss points
corresponding to ${\cal Z}_1$ are depicted with `$\bullet$' and the Gauss points corresponding to ${\cal Z}_2\setminus {\cal Z}_1$ are depicted with
heavy '$\star$'.
\end{remark}

\medskip

We next explain how to determine $\Pi v_{\cal P}$ by \eqref{mapcond}.
 For any given $v_{\cal P}$, $\Pi v_{\cal P}=0$ on the boundary control volumes,
 namely, $(\Pi v_{\cal P})_{i,k;j,l}=0, \forall (i,k;j,l)\in {\cal Z}_0\setminus {\cal Z}_1$. In the first step, let $(i,k;j,l)=(1,1;1,1)$ in \eqref{mapcond} to obtain
 \[
 (\Pi v_{\cal P})_{1,1;1,1}=A_{1,1}^xA_{1,1}^y \frac{\partial^2 v_{\cal P}}{\partial x\partial y}(g_{1,1}^x, g^y_{1,1}).
 \]
 Once $(\Pi v_{\cal P})_{1,1;1,1}$ is obtained, we then let $(j,l)\in \bZ_n\times\bZ_{r_j}\setminus \{(1,1)\}$ in \eqref{mapcond} to obtain
\[
 (\Pi v_{\cal P})_{1,1;j,l}=(\Pi v_{\cal P})_{1,1;j,l-1}+A_{1,1}^xA_{j,l}^y \frac{\partial^2 v_{\cal P}}{\partial x\partial y}(g_{1,1}^x, g^y_{j,l}),
\]
where we  again used the fact $\Pi v_{\cal P}=0$ on the boundary volumes.
That is, $(\Pi v_{\cal P})_{1,1;j,l}$ can be successively calculated for all $(j,l)\in {\bZ}_n\times\bZ_{r_j}$.
In the same way, we can  calculate  $(\Pi v_{\cal P})_{i,k;1,1}$ for all $(i,k)\in \bZ_m\times\bZ_{r_i}$
 successively.
In the second step, we use \eqref{mapcond} to compute all $(\Pi v_{\cal P})_{1,2;j,l}$ and all $(\Pi v_{\cal P})_{i,k;1,2}$
for all $(j,l)\in \bZ_n\times\bZ_{r_j}\setminus \{(1,1),{1,2}\}$ and all $(i,k)\in \bZ_m\times\bZ_{r_i}\setminus \{(1,1),{1,2}\}$.
In the $p$-th step, $p=2,\ldots, \min(mr-1,nr-1)$, we use \eqref{mapcond} to compute all $(\Pi v_{\cal P})_{i_p,k_p;j,l}$ and all $(\Pi v_{\cal P})_{i,k;i_p,k_p}$
for all $(j,l)\in \in \bZ_n\times\bZ_{r_j}\setminus \{(1,1),{1,2}, \ldots (i_p,k_p-1)\}$ and all $(i,k)\in \bZ_m\times\bZ_{r_i}\setminus \{(1,1),{1,2},\ldots,(i_p,j_p-1\}$.
Finally, we obtain   $(\Pi v_{\cal P})_{i,k;j,l}$ for all $(i,k;j,l)\in {\cal Z}_1$.

Next, we show that
\eqref{mapcond} holds for all $(i,k;j,l)\in {\cal Z}_2.$  In fact,
since $v_{\cal P}=0$ on the boundary $\partial\Omega$,
\[
\frac{\partial v_{\cal P}}{\partial x}(x, c)=\frac{\partial v_{\cal P}}{\partial x}(x, d)=0, \forall x\in[a,b],
\]
then for all $(i,k)\in \bZ_{m,r}$,
\[
\sum_{(j,l)\in \bZ_{n,r}} A_{j,l}^y\frac{\partial^2 v_{\cal P}}{\partial x\partial y}(g_{i,k}^x, g^y_{j,l})=\int_c^d\frac{\partial^2 v_{\cal P}}{\partial x\partial y}(g_{i,k}^x, y)dy=0.
\]
In other words,
\[
A_{n,r}^y\frac{\partial^2 v_{\cal P}}{\partial x\partial y}(g_{i,k}^x, g^y_{n,r})=-\sum_{(j,l)\in \bZ_{n,r_j^y}} A_{j,l}^y\frac{\partial^2 v_{\cal P}}{\partial x\partial y}(g_{i,k}^x, g^y_{j,l}).
\]
Consequently,
\begin{eqnarray*}
A_{i,k}^xA_{n,r}^y\frac{\partial^2 v_{\cal P}}{\partial x\partial y}(g_{i,k}^x, g^y_{n,r})&=&-\sum_{(j,l)\in \bZ_{n,r_j^y}} \lfloor \Pi v_{\cal P}\rfloor_{i,k;j,l}\\
&=&-\sum_{(j,l)\in \bZ_{n,r_j^y}} [[(\Pi v_{\cal P})_{i,k;j,l}]^x]^y\\
&=&-[\Pi v_{\cal P}]^x_{i,k;n,r}=\lfloor\Pi v_{\cal P}\rfloor_{i,k;n,r}.
\end{eqnarray*}
By the same reasoning, for all $(j,l)\in \bZ_{n,r}$,
\[
\lfloor\Pi v_{\cal P}\rfloor_{m,r;j,l}=A_{m,r}^xA_{j,l}^y\frac{\partial^2 v_{\cal P}}{\partial x\partial y}(g_{m,r}^x, g^y_{j,l}).
\]
Namely, \eqref{mapcond} holds for all $(i,k;j,l)\in {\cal Z}_2$.


\medskip

\begin{lemma}\label{lm:ub}If  ${\cal P}$ is shape regular, then for any $v_{\cal P}\in U_{\cal P}^r$,
\begin{equation}\label{ub}
|\Pi v_{\cal P}|_{\cal P'}\lesssim |v_{\cal P}|_1,
\end{equation}
where the hidden constant depends only on $r$.
\end{lemma}
\begin{proof}By the definition of the semi-norm $|\cdot|_{\cal P'}$, we have
\begin{equation}\label{def'}
|\Pi v_{\cal P}|^2_{\cal P'}=\sum_{(i,k;j,l)\in {\cal Z}_3} \left([(\Pi v_{\cal P})_{i,k;j,l}]^x\right)^2+\sum_{(i,k;j,l)\in {\cal Z}_4} \left([(\Pi v_{\cal P})_{i,k;j,l}]^y\right)^2.
\end{equation}
We next estimate $[(\Pi v_{\cal P})_{i,k;j,l}]^x$.
For all given $(i,k;j)\in\bZ_m\times\bZ_r\times\bZ_n$, the function $\frac{\partial^2 v_{\cal P}}{\partial x\partial y}(g_{i,k}^x, \cdot)$ is a polynomial of degree $r-1$ in the
interval $[y_{j-1}, y_j]$, therefore,
\begin{eqnarray*}
\sum_{l\in \bZ_r} A_{j,l}^y \frac{\partial^2 v_{\cal P}}{\partial x\partial y}(g_{i,k}^x, g^y_{j,l})&=&\int_{y_{j-1}}^{y_{j}}\frac{\partial^2 v_{\cal P}}{\partial x\partial y}(g_{i,k}^x, y)dy\\
&=&\frac{\partial v_{\cal P}}{\partial x}(g_{i,k}^x, y_{j})-\frac{\partial v_{\cal P}}{\partial x}(g_{i,k}^x, y_{j-1}).
\end{eqnarray*}
Noticing that
\[
\frac{\partial v_{\cal P}}{\partial x}(g_{i,k}^x, c)=0,
\]
we have
\begin{eqnarray*}
\sum_{j'\in \bZ_j}\sum_{l\in \bZ_r} A_{j',l}^y \frac{\partial^2 v_{\cal P}}{\partial x\partial y}(g_{i,k}^x, g^y_{j',l})
&=&\frac{\partial v_{\cal P}}{\partial x}(g_{i,k}^x, y_{j}).
\end{eqnarray*}
Consequently, for all $(i,k;j,l)\in {\cal Z}_3$
\begin{eqnarray}\nonumber
[\Pi v_{\cal P}]^x_{i,k;j,l}&=&\sum_{(j',l')\in\bZ_j\times\bZ_{r_{j,l}}}\lfloor \Pi v_{\cal P}\rfloor_{i,k;j',l'}\\ \nonumber
&=&A_{i,k}^x\sum_{(j',l')\in\bZ_j\times\bZ_{r_{j,l}}}A_{j',l'}^y \frac{\partial^2 v_{\cal P}}{\partial x\partial y}(g_{i,k}^x, g^y_{j',l'})\\ \label{estt}
&=&A_{i,k}^x\left(\frac{\partial v_{\cal P}}{\partial x}(g_{i,k}^x, y_{j-1})+\sum_{l'\in\bZ_l}A_{j,l'}^y \frac{\partial^2 v_{\cal P}}{\partial x\partial y}(g_{i,k}^x, g^y_{j,l'})\right),
\end{eqnarray}
where $r_{j,l}=r$, if $j'<j$ and  $r_{j,l}=l$, if $j'=j$.

We first estimate  $\frac{\partial v_{\cal P}}{\partial x}(g_{i,k}^x, y_{j-1})$. Since $\frac{\partial v_{\cal P}}{\partial x}(g_{i,k}^x, \cdot)\in \bP_{r-1}(y_{j-1},y_j)$, by the inverse inequality, there holds
\[
\left\|\frac{\partial v_{\cal P}}{\partial x}(g_{i,k}^x, \cdot)\right\|_{L_\infty([y_{j-1},y_j])}
\lesssim h^{-\frac12}\left\|\frac{\partial v_{\cal P}}{\partial x}(g_{i,k}^x, \cdot)\right\|_{L_2([y_{j-1},y_j])},
\]
where the hidden constant depends only on $r$. Then
\begin{equation}\label{est1}
\left|\frac{\partial v_{\cal P}}{\partial x}(g_{i,k}^x, y_{j-1})\right| \lesssim h^{-\frac12}\left\|\frac{\partial v_{\cal P}}{\partial x}(g_{i,k}^x, \cdot)\right\|_{L_2([y_{j-1},y_j])}.
\end{equation}
We now estimate the second term of \eqref{estt}. By the Cauchy-Schwartz inequality,
\begin{eqnarray*}
\left(\sum_{l'\in\bZ_l}A_{j,l'}^y \frac{\partial^2 v_{\cal P}}{\partial x\partial y}(g_{i,k}^x, g^y_{j,l'})\right)^2
&\le&
r\sum_{l'\in\bZ_r} (A_{j,l'}^y)^2 \left(\frac{\partial^2 v_{\cal P}}{\partial x\partial y}(g_{i,k}^x, g^y_{j,l'})\right)^2 \\
&\le&
rh_j^y\sum_{l'\in\bZ_r} A_{j,l'}^y \left(\frac{\partial^2 v_{\cal P}}{\partial x\partial y}(g_{i,k}^x, g^y_{j,l'})\right)^2\\
&=&
rh_j^y\int_{y_{j-1}}^{y_j}  \left(\frac{\partial^2 v_{\cal P}}{\partial x\partial y}(g_{i,k}^x, y)\right)^2 dy.
\end{eqnarray*}
Again since $\frac{\partial v_{\cal P}}{\partial x}(g_{i,k}^x, \cdot)\in \bP_{r-1}{(y_{j-1}, y_j)}$, we have the inverse inequality
\[
 \left\|\frac{\partial^2 v_{\cal P}}{\partial y\partial x}(g_{i,k}^x, \cdot)\right\|_{L^2(y_{j-1}, y_j)} \lesssim h^{-1} \left\|\frac{\partial v_{\cal P}}{\partial x}(g_{i,k}^x, \cdot)\right\|_{L^2(y_{j-1},y_j)},
\]
where the hidden constant again depends only on $r$. Consequently,
\begin{equation}\label{est2}
\left(\sum_{l'\in\bZ_l}A_{j,l'}^y \frac{\partial^2 v_{\cal P}}{\partial x\partial y}(g_{i,k}^x, g^y_{j,l'})\right)^2\lesssim h^{-1}\int_{y_{j-1}}^{y_j} \left(\frac{\partial v_{\cal P}}{\partial x}(g_{i,k}^x, \cdot)\right)^2dy.
\end{equation}
Substituting \eqref{est1} and \eqref{est2} into \eqref{estt} and noticing the fact $A_{i,k}^xh^{-1}\le 1$, we obtain that for all $(i,k;j,l)\in {\cal Z}_3$
\[
\left([\Pi v_{\cal P}]^x_{i,k;j,l}\right)^2\lesssim A_{i,k}^x\int_{y_{j-1}}^{y_j}  \left(\frac{\partial v_{\cal P}}{\partial x}(g_{i,k}^x, y)\right)^2 dy.
\]
Consequently,
\[
\sum_{(i,k;j,l)\in {\cal Z}_3} \left([(\Pi v_{\cal P})_{i,k;j,l}]^x\right)^2\lesssim \sum_{(i,k)\in \bZ_{m,r}}A_{i,k}^x\int_{c}^{d}  \left(\frac{\partial v_{\cal P}}{\partial x}(g_{i,k}^x, y)\right)^2 dy.
\]
Since the function $\int_{c}^{d}  \left(\frac{\partial v_{\cal P}}{\partial x}(\cdot, y)\right)^2 dy\in \bP_{2r-2}(x_{i-1},x_i)$ for all $i\in \bZ_m$, we obtain
\begin{equation}\label{pix}
\sum_{(i,k;j,l)\in {\cal Z}_3} \left([(\Pi v_{\cal P})_{i,k;j,l}]^x\right)^2\lesssim \int_a^b\int_{c}^{d}  \left(\frac{\partial v_{\cal P}}{\partial x}(x, y)\right)^2 dydx.
\end{equation}
By the same arguments,
\begin{equation}\label{piy}
\sum_{(i,k;j,l)\in {\cal Z}_4} \left([(\Pi v_{\cal P})_{i,k;j,l}]^y\right)^2\lesssim \int_a^b\int_{c}^{d}  \left(\frac{\partial v_{\cal P}}{\partial y}(x, y)\right)^2 dxdy.
\end{equation}
Therefore, by \eqref{def'},
\begin{eqnarray*}
|\Pi v_{\cal P}|^2_{\cal P'}&\lesssim& \int_a^b\int_c^d  \left(\frac{\partial v_{\cal P}}{\partial x}(x, y)\right)^2+ \left(\frac{\partial v_{\cal P}}{\partial y}(x, y)\right)^2 dxdy,
\end{eqnarray*}
from which the inequality \eqref{ub} follows.
\end{proof}

\bigskip

\begin{lemma}\label{le:infsup}If $\alpha$ is piecewise constant with respect to  ${\cal P}$, then
\begin{equation}\label{infsup1}
   a_{\cal P} (v_{\cal P} ,\Pi v_{\cal P} )\ge \alpha_0|v_{\cal P} |_1^2,\quad \forall v_{\cal P}\in U^r_{\cal P}.
\end{equation}
\end{lemma}
\begin{proof}We define two functions for all  $(x,y)\in \Omega$ by
\[
v^1(x,y)=\int_c^y \alpha(x,y')\frac{\partial v_{\cal P}}{\partial x}(x,y') dy',\quad v^2(x,y)=\int_a^x \alpha(x',y)\frac{\partial v_{\cal P}}{\partial y}(x',y) dx'.
\]
A straightforward calculation yields that
\begin{eqnarray*}
   a_{\cal P}(v_{\cal P},\Pi v_{\cal P})
  &=& -\sum_{(i,k;j,l)\in {\cal Z}_2}\lfloor \Pi v_{\cal P} \rfloor_{i,k;j,l}\left( v^1(g^x_{i,k}, g^y_{j,l})+ v^2(g^x_{i,k}, g^y_{j,l})\right).\\
\end{eqnarray*}
Noticing that \eqref{mapcond} holds for all $(i,k;j,l)\in {\cal Z}_2$, we obtain
\begin{eqnarray*}
   a_{\cal P}(v_{\cal P},\Pi v_{\cal P})
    &=&I_1+I_2,
\end{eqnarray*}
where
\[
I_1=-\sum_{(i,k;j,l)\in {\cal Z}_2}A_{i,k}^xA_{j,l}^y \frac{\partial^2 v_{\cal P}}{\partial x\partial y}(g_{i,k}^x, g^y_{j,l}) v^1(g^x_{i,k}, g^y_{j,l}),
\]
and
\[
I_2=-\sum_{(i,k;j,l)\in {\cal Z}_2}A_{i,k}^xA_{j,l}^y \frac{\partial^2 v_{\cal P}}{\partial x\partial y}(g_{i,k}^x, g^y_{j,l}) v^2(g^x_{i,k}, g^y_{j,l}).
\]
We next estimate $I_1$. To this end, for all $(i,k;j)\in \bZ_m\times\bZ_r\times\bZ_n$, we denote by
\begin{align*}
Err_j(g_{i,k}^x)=\int_{y_{j-1}}^{y_j}\frac{\partial^2 v_{\cal P}}{\partial x\partial y}(g_{i,k}^x,y)v^1(g^x_{i,k}, y)dy-\sum_{l\in \bZ_r} A_{j,l}^y \frac{\partial^2 v_{\cal P}}{\partial x\partial y}(g_{i,k}^x, g^y_{j,l})v^1(g^x_{i,k}, g^y_{j,l})
\end{align*}
the error of the Gauss quadrature of the function $\frac{\partial^2 v_{\cal P}}{\partial x\partial y}(g_{i,k}^x,y)v^1(g^x_{i,k}, \cdot)$ in the interval $[y_{j-1},y_j]$.
Moreover, let
\[
Res=\sum_{(i,k;j)\in \bZ_m\times\bZ_r\times\bZ_n}A_{i,k}^x Err_j(g_{i,k}^x).
\]
With this notation,
\begin{eqnarray*}
I_1&=&-\sum_{(i,k)\in \bZ_m\times\bZ_r}A_{i,k}^x\int_c^d\frac{\partial^2 v_{\cal P}}{\partial x\partial y}(g_{i,k}^x,y)v^1(g^x_{i,k}, y)dy+Res\\
&=&\sum_{(i,k)\in \bZ_m\times\bZ_r}A_{i,k}^x\int_c^d\alpha(g^x_{i,k},y)\left(\frac{\partial v_{\cal P}}{\partial x}(g^x_{i,k}, y)\right)^2dy+Res.
\end{eqnarray*}
Since $\alpha$ is a constant in each $\tau\in {\cal P}$,
\[
\int_c^d\alpha(\cdot,y)\left(\frac{\partial v_{\cal P}}{\partial x}(\cdot, y)\right)^2dy\in \bP_{2r-2}([x_{i-1}, x_i]), \forall  i\in \bZ_m,
\]
then
\begin{equation}\label{apvp}
I_1=\int_a^b\int_c^d\alpha(x,y)\left(\frac{\partial v_{\cal P}}{\partial x}(x, y)\right)^2dydx+Res.
\end{equation}
We next estimate $Res$. Using the fact that $\alpha$ is piecewise constant with respect to ${\cal P}$, we have
\[
\frac{\partial^2 v_{\cal P}}{\partial x\partial y}(g_{i,k}^x,\cdot)v^1(g^x_{i,k}, \cdot)\in \bP_{2r}([y_{j-1},y_j]), \forall j\in \bZ_n.
\]
Then for all $y\in [y_{j-1},y_j]$,
\[
\frac{\partial^{(2r)}}{\partial y^{(2r)}}\left(\frac{\partial^2 v_{\cal P}}{\partial x\partial y}(g_{i,k}^x,y)v^1(g^x_{i,k}, y)\right)=\alpha(g_{i,k}^x,y) (2r)!\times\frac{r}{r+1}\times a_{i,k,j}^2 \ge 0.
\]
where $a_{i,k;j}$ is the leading coefficient of the polynomial $\frac{\partial v_{\cal P}}{\partial x}(g_{i,k}^x,y)$ in $[y_{j-1}, y_j]$.
Consequently,
\[
Err_j(g_{i,k}^x)\ge 0, \forall (i,k)\in \bZ_m\times\bZ_r\times\bZ_n,
\]
and thus \[Res\ge 0.\]
In summary,
\begin{equation}\label{y}
I_1\ge \alpha_0\int_a^b\int_c^d\left(\frac{\partial v_{\cal P}}{\partial x}(x, y)\right)^2dydx.
\end{equation}
By the same arguments,
\begin{equation}\label{x}
I_2\ge \alpha_0\int_c^d\int_a^b\left(\frac{\partial v_{\cal P}}{\partial y}(x, y)\right)^2dxdy.
\end{equation}
Combining \eqref{y} and\eqref{x}, the inequality \eqref{infsup1} follows.
\end{proof}

\bigskip
Summarizing the above two lemmas, we obtain the following inf-sup property.

\begin{theorem}Let ${\cal P}$ be a shape regular and quasi-uniform partition with the meshsize $h$.
If the coefficient $\alpha$ is  piecewise constant, then the inf-sup property
\begin{equation}\label{infsup}
\inf_{v_{\cal P}\in U^r_{\cal P}}\sup_{w_{\cal P'} \in V_{\cal P'} }
\frac{a_{\cal P} (v_{\cal P} ,w_{\cal P'} )}{|v_{\cal P} |_1 |w_{\cal P'}|_{\cal P'}}\gtrsim 1
\end{equation}
holds for all $h$. If $\alpha$ is piecewise continuous, then \eqref{infsup} holds when the meshsize $h$ is sufficiently small.
\end{theorem}
\begin{proof}When $\alpha$ is piecewise constant, by \eqref{infsup1} and \eqref{ub}, for any $v_{\cal P}\in U^r_{\cal P}$,
\[
\sup_{w_{\cal P'}\in V_{\cal P'}} \frac{a_{\cal P}(v_{\cal P},w_{\cal P'})}{|w_{\cal P'}|_{\cal P'}}
\ge \frac{a_{\cal P}(v_{\cal P},\Pi v_{\cal P})}{|\Pi v_{\cal P}|_{\cal P'}}\ge \frac{\alpha_0|v_{\cal P}|_1^2}{|\Pi v_{\cal P}|_{\cal P'}}\gtrsim |v_{\cal P}|_{1}.
\]
The inf-sup condition \eqref{infsup} is proved.

When $\alpha$ is only piecewise continuous, let
\[
\bar{\alpha}(x,y)=\frac{1}{|\tau|}\int_\tau \alpha(x,y)dxdy, \forall (x,y)\in\tau\in {\cal P}
\]
and denote  the piecewise modulus of continuity of $\alpha$ by
\begin{equation*}\label{modulusofcontinuity}
m_{\cal P}(\alpha,h)=\sup\{\big|\alpha({\bf x}_1)-\alpha({\bf x}_2)\big|:\big|{\bf x}_1-{\bf x}_2\big|\le h,\
\forall\ {\bf x}_1, {\bf x}_2\in \tau, \forall\ \tau\in {\cal P}\}.
\end{equation*}
The fact that $\alpha$ is piecewise constant implies that  $m_{\cal
P}(\alpha,h)$ converges to 0 when $h$ goes to 0.
Since $\bar{\alpha}$ is  piecewise constant,  by Lemma \ref{le:infsup},
\begin{eqnarray*}
\bar{a}_{\cal P}(v_{\cal P},\Pi v_{\cal P})
&:=&  \sum_{E\in \E_{\cal P'}}[\Pi v_{\cal P}]\int_E \bar{\alpha}\frac{\partial v}{\partial {\bf n}}ds\ge \alpha_0 |v_{\cal P}|^2_1.
\end{eqnarray*}
On the other hand, by the same arguments in Theorem \ref{thm:continuity}, we have
\[
\left|a_{\cal P}(v_{\cal P},\Pi v_{\cal P})-\bar{a}_{\cal P}(v_{\cal P},\Pi v_{\cal P})\right|\lesssim m_{\cal P}(\alpha,h)|v_{\cal P}|^2_1,
\]
Then when $h$ is sufficiently small,
$$
a_{\cal P}(v_{\cal P},\Pi v_{\cal P})\ge \left(\alpha_0-cm_{\cal
P}(\alpha,h)\right)\big|v_{\cal P}\big|^2_{1,\Omega}\geq
\frac{\alpha_0}{2}\big|v_{\cal P}\big|^2_{1,\Omega}.$$
The inf-sup condition \eqref{infsup} is proved.
\end{proof}

%

\subsection{$H^1$ error estimate}
We begin with some preparations. First, by the inverse inequality,
\[
\big|v_{\cal P}\big|_{\cal P}\sim \big|v_{\cal P}\big|_1, \forall v_{\cal P}\in U_{\cal P}^r.
\]
With this equivalence, we deduce from the inf-sup condition \eqref{infsup} that
\begin{equation}\label{infsup2}
\inf_{v_{\cal P}\in U_{\cal P}^r}\sup_{w_{\cal P'} \in V_{\cal P'} }
\frac{a_{\cal P} (v_{\cal P} ,w_{\cal P'} )}{|v_{\cal P} |_{\cal P} |w_{\cal P'}|_{\cal P'}}\ge c_0,
\end{equation}
where the  constant $c_0=c_0(r)$ depends only on $r$.

Secondly, we denote by $L_1,\ldots,L_{r-1}$, the zeros of $L'_{r}(x)$, where
  $L_r(x)$ is the Legendre polynomial of order $r$ in the interval $[-1,1]$.  Moreover, we denote $L_0=-1, L_r=1$.
  The family of the points $L_j, j=0,\ldots,r$  are called the Lobatto points of degree $r$.

The Lobatto points on the rectangle $\tau_{i;j}, (i,j)\in\bZ_{m,n}$ are defined as
\[
   l_{i,k;j,l}=\left(\frac 12(x_i+x_{i-1}+h_i^x L_k), \frac 12(y_j+y_{j-1}+h_j^yL_l)\right),\ \ (k,l)\in \bZ_{r,r}.
\]
   Let $u_I\in U^r_{\cal P}$ be the interpolation of $u$ that satisfies
\[
u_I(l_{i,k;j,l})= u(l_{i,k;j,l}), (i,k;j,l)\in {\cal Z}_0.
\]
Note that this kind of interpolation has been used in the literature for superconvergence analysis, see, e.g., \cite{Zhang1, Zhang2} for the one-dimensional situation.


\medskip

We are now ready to present our main result.
\begin{theorem}Let $u$ be the solution of \eqref{eq:2.1} and \eqref{eq:2.2}, $u_{\cal P}$ the solution of \eqref{bilinear}. Then for sufficiently small $h$,
\begin{equation}\label{totalestimate}
\big|u-u_{\cal P}\big|_{\cal P}\le \frac{M}{c_0} \inf_{v_{\cal P}\in U_{{\cal P}}}\big|u-v_{\cal P}\big|_{\cal P}.
\end{equation}
Consequently, if $u\in H^{r+1}(\Omega)$,
\begin{equation}\label{optimalh1}
   |u-u_{\cal P}|_{1}\lesssim h^r|u|_{r+1},
\end{equation}
  where the hidden  constant is independent of the mesh size $h$.
\end{theorem}
\begin{proof}By \eqref{exact}, \eqref{bounded} and the inf-sup condition \eqref{infsup2}, for all $v_{\cal P}\in U_{{\cal P}}$,
there holds
\[
\big|u-u_{\cal P}\big|_{\cal P}\le \big|u-v_{\cal P}\big|_{\cal P}+\big|v_{\cal P}-u_{\cal P}\big|_{\cal P}
\le \big(1+\frac{M}{c_0}\big)\big|u-v_{\cal P}\big|_{\cal P}.
\]
Using a technique in Xu and Zikatanov
(\cite{XuZikatanov2003}), the constant $1+\frac{M}{c_0}$ in the above inequality can be reduced to
$\frac{M}{c_0}$. That is, \eqref{totalestimate} holds.

We conclude from the definition of $|\cdot|_{\cal P}$ and \eqref{totalestimate} that
\[
  |u-u_{\cal P}|_1\le |u-u_{\cal P}|_{\cal P}\lesssim \inf\limits_{v_{\cal P}\in U^r_{\cal P}}|u-v_{\cal P}|_{\cal P}.
\]
  Note that
\[
   \inf\limits_{v_{\cal P}\in U^r_{\cal P}}|u-v_{\cal P}|_{\cal P}\le |u-u_I|_{1}+h|u-u_I|_{2},
\]
where $u_I$ is the Lagrange interpolation of $u$ at the Lobatto points in the trial space $U^r_{\cal P}$.
By the standard approximation theory, we obtain the estimate \eqref{optimalh1}.
\end{proof}

\medskip

\subsection{Superconvergence and $L^2$ error estimates}
We first present a superconvergence result and then use it to establish our $L^2$ error estimate.
\begin{theorem}\label{weakinterpolation}   Assume that $u\in H_0^1(\Omega)\cap H^{r+2}(\Omega)$ is the solution of \eqref{eq:2.1}-\eqref{eq:2.2}, and
  $u_{\cal P}$ is the solution of the FV scheme \eqref{bilinear}. Then for all $w_{\cal P'}\in V_{\cal P'}$,
\begin{equation}\label{b2}
   |a_{\cal P}(u-u_I,w_{\cal P'})|\lesssim h^{r+1}|u|_{r+2,{\cal P}}|w_{\cal P'}|_{\cal P'},
\end{equation}
where $|u|_{r+2,{\cal P}}=\left(\sum_{\tau\in{\cal P}}|u|^2_{r+2,\tau}\right)^{1\over 2}$.
Consequently,
\begin{equation}\label{supercon}
   |u_I-u_{\cal P}|_1\lesssim h^{r+1}|u|_{r+2,{\cal P}}.
\end{equation}
\end{theorem}
\begin{proof}
We can derive the following inequality by the standard superconvergence argument, see, e.g., \cite{Zhu.QD;LinQ1989},
\begin{equation}\label{stresspoint}
   \left\|\frac{\partial (u-u_{I})}{\partial x}(g^x_{i,k},\cdot)\right\|_{L^\infty[g^y_{j,l},g^y_{j,l+1}]},    \left\|\frac{\partial (u-u_{I})}{\partial y}(\cdot,g^y_{j,l})\right\|_{L^\infty[g^x_{i,k},g^x_{i,k+1}]} \lesssim h^r|u|_{r+2,1,\tilde{\tau}_{i,k;j,l}},
\end{equation}
where $\tilde{\tau}_{i,k;j,l}=[g^x_{i,k-1},g^x_{i,k+1}]\times [g^y_{j,l-1},g^y_{j,l+1}]$. It follows from \eqref{bilinear1} that
\begin{eqnarray*}
  \big|a_{\cal P}(u-u_I,w_{\cal P'})|&\le& \|\alpha\|_\infty|w_{\cal P'}|_{\cal P'} \left(\sum_{E\in \E_{\cal P'}} h_E\int_E \left(\frac{\partial (u-u_I)}{\partial {\bf n}}\right)^2ds\right)^{\frac12}\\
     &\lesssim & h^{r+1}|w_{\cal P'}|_{\cal P'}|u|_{r+2,{\cal P}},
\end{eqnarray*}
where in the last step we have used \eqref{stresspoint} and the fact $|u|_{r+2,1,\tau_i}\lesssim h_i^{\frac 12}|u|_{r+2,\tau_i}$ and

We next show \eqref{supercon}. By the inf-sup condition \eqref{infsup},
\[
   |u_I-u_{\cal P}|_1\lesssim \sup\limits_{w_{\cal P'}\in V_{\cal P'}}
  \frac {a_{\cal P}(u_{\cal P}-u_I,w_{\cal P'})}{|w_{\cal P'}|_{\cal P'}}
   =\sup\limits_{w_{\cal P'}\in V_{\cal P'}}\frac {a_{\cal P}(u-u_I,w_{\cal P'})}{|w_{\cal P'}|_{\cal P'}}.
\]
Combining the above inequality with \eqref{b2}, we derive \eqref{supercon}.
\end{proof}

As a direct consequence of the superconvergence result \eqref{supercon}, we have the following $L^2$ estimate.
\begin{theorem}
   Assume that $u\in H_0^1(\Omega)\cap H^{r+2}(\Omega)$ is the solution of \eqref{eq:2.1}-\eqref{eq:2.2}, and
  $u_{\cal P}$ is the solution of the FV scheme \eqref{bilinear}, then there holds
\begin{equation}\label{optimall2}
 \|u-u_{\cal P}\|_{0}\lesssim h^{r+1}\|u\|_{r+2}.
\end{equation}
\end{theorem}
\begin{proof}
By the triangle inequality,
\[
\|u-u_{\cal P}\|_0\le \|u-u_I\|_0+\|u_{\cal P}-u_I\|_0
\]
where $u_I$ is the interpolation of $u$ given in the previous subsection.

Since $u_I=u_{\cal P}=0$ on $\partial\Omega$,  we have by the Poincar\'e inequality that
\[
\|u_{\cal P}-u_I\|_0\lesssim  |u_{\cal P}-u_I|_1 \lesssim h^{r+1}|u|_{r+2}.
\]
Moreover,
\[
\|u-u_I\|_0\lesssim h^{r+1}\|u\|_{r+1}\le h^{r+1}\|u\|_{r+2}.
\]
Then we have \eqref{optimall2}.
\end{proof}

\begin{remark}
In the above $L^2$ error estimate, we do not need to use the so-called Aubin-Nitsche techniques.
\end{remark}

\section{Numerical Results}
In this section, we present a numerical example to validate the theoretical results proved in previous sections.
We consider \eqref{eq:2.1}-\eqref{eq:2.2} with $\alpha=1$ and $\Omega=[0,1]^2$. We choose the right-hand side function
\[
f(x,y)=2\pi^2\sin\pi x\sin\pi y, (x,y)\in [0,1]^2
\]
which allows the exact solution
\[
u(x,y)=sin\pi x\sin\pi y, (x,y)\in [0,1]^2.
\]
We use FV schemes \eqref{bilinear} with $r=2,3,4,5$ to compute FVM approximate solutions
of $u$.  The partition ${\cal P}_j$,
$j=1,\dots,6$, are obtained by uniformly refining the unite square $[0,1]^2$. For simplicity, we write $u_j$, instead of $u_{\cal P_j}$, as the finite volume solution $u_{\cal P_j}\in U^r_{\cal P_j}$.

The numerical results are demonstrated in Figures 1, 2, 3, and 4, respectively. In these figures, we depict $n_j^{-r}$ by the solid curve with `$\Box$' and depict $n_j^{-(r+1)}$ by the dash line, where $n_j$ is the number of subintervals of the underlying grid ${\cal P}_j$.
We depict $|u-u_j|_{H^1}$ by the solid curve with `$\star$', $\|u-u_j\|_{L^2}$  by the solid curve with `$\triangle$', and $|u_I-u_j|_{H^1}$ by the solid curve with `$*$'.
We observe that $|u-u_j|_{H^1}$ decays with the convergence rate $n_j^{-r}$ which supports our theory \eqref{optimalh1}. We also observe that
both $\|u-u_j\|_{L^2}$ and $|u_I-u_j|_{H^1}$ decay with $n_j^{-(r+1)}$. These facts support our $L^2$ estimate \eqref{optimall2} and superconvergene result \eqref{supercon}.

%
%

\section{Conclusions and Future Works}
The design and analysis of high-order FV schemes are challenging tasks. This paper is the second one
in its series that attempts to set up a mathematical foundation for a family of high order FV schemes.
In a previous work (\cite{CaoLiZhangZou2012}), we studied convergence and superconvergence
properties of FV schemes of any order for the one-dimensional elliptic equations.
The higher dimensional case is fundamentally different from, and much more complicated than the one dimensional case. In this article, we only report our results for rectangular meshes. For the family of higher-order FV schemes discussed in this paper, we obtained almost the same basic theoretical results as for the counterpart higher-order FEM.
The results for more general meshes and more general equations will be reported in a forthcoming paper \cite{ZhangZou2012}.

\section*{Acknowledgement}
The authors would like to thank a Ph.D. student, Waixiang Cao for her assistance in the numerical experiments.

\bigskip
\bigskip
\end{document}